\documentclass{amsart}
\usepackage{graphicx}
\usepackage{epstopdf}
\usepackage{pstricks}
\usepackage[mathscr]{eucal}
\theoremstyle{plain}
\newtheorem{prop}{Proposition}[section]
\newtheorem{coro}[prop]{Corollary}

\newtheorem{conj}[prop]{Conjecture}
\newtheorem{lemm}[prop]{Lemma}

\newtheorem{thm}[prop]{Theorem}

\theoremstyle{definition}
\newtheorem{defn}[prop]{Definition}

\newtheorem{rem}[prop]{Remark}

\DeclareMathOperator{\sign}{sign}
\DeclareMathOperator{\br}{br}

\def\mcg#1;#2{\Gamma_{#1,#2}}
\def\fg#1;#2{\Pi_{#1,#2}}
\def\tb#1;#2{\mathscr{K}_{\frac{#1}{#2}}}

\begin{document}

\title[On The Jones Polynomial of Quasi-alternating Links]
{On The Jones Polynomial of Quasi-alternating Links}

\thanks{The first author was supported by a  research grant from United Arab Emirates University, UPAR grant $\#$G00002650.}

\author{Nafaa Chbili}
\address{Department of Mathematical Sciences\\ College of Science\\ UAE University \\ 15551 Al Ain, U.A.E.}
\email{nafaachbili@uaeu.ac.ae}
\urladdr{http://faculty.uaeu.ac.ae/nafaachbili}

\author{Khaled Qazaqzeh}
\address{Department of Mathematics\\ Faculty of Science \\ Kuwait University\\
P. O. Box 5969\\ Safat-13060, Kuwait, State of Kuwait}
\email{khaled@sci.kuniv.edu.kw}

\date{23/05/2018}

\begin{abstract}
We prove that twisting any quasi-alternating link $L$ with  no gaps in its Jones polynomial $V_L(t)$  at the crossing where it is
quasi-alternating produces a link $L^{*}$  with  no gaps in its Jones polynomial $V_{L^*}(t)$.
This leads us to conjecture that the Jones polynomial of any prime quasi-alternating link, other than
$(2,n)$-torus links, has no gaps. This would give a new property of quasi-alternating links
and a simple obstruction criterion for a link to be quasi-alternating. We prove that the conjecture
holds for quasi-alternating Montesinos links as well as quasi-alternating links with braid index 3.

\end{abstract}

\maketitle

\section{introduction}

The class of quasi-alternating links was introduced first by
Ozsv$\acute{a}$th and Szab$\acute{o}$ in \cite{OS} as a natural generalization of the class of alternating links.
This class is defined recursively as follows:
\begin{defn}\label{def}
The set of quasi-alternating links $\mathcal{Q}$  is the smallest
set satisfying the following properties.
\begin{itemize}
  \item The unknot belongs to $\mathcal{Q}$.
  \item If $L$ is a link with a diagram $D$ containing a crossing $c$ such that
\begin{enumerate}
\item both smoothings of the diagram $D$ at the crossing $c$, $L_{0}$ and $L_{\infty}$
as given in Figure \ref{figure} belong to $\mathcal{Q}$,
\item $\det(L_{0}), \det(L_{\infty}) \geq 1$,
\item $\det(L) = \det(L_{0}) + \det(L_{\infty})$; then $L$ is in $\mathcal{Q}$
and in this case we say that  $L$ is quasi-alternating at the
crossing $c$ with quasi-alternating diagram $D$.
\end{enumerate}
\end{itemize}
\end{defn}

\begin{figure} [h]
\begin{center}
\includegraphics[scale=0.35]{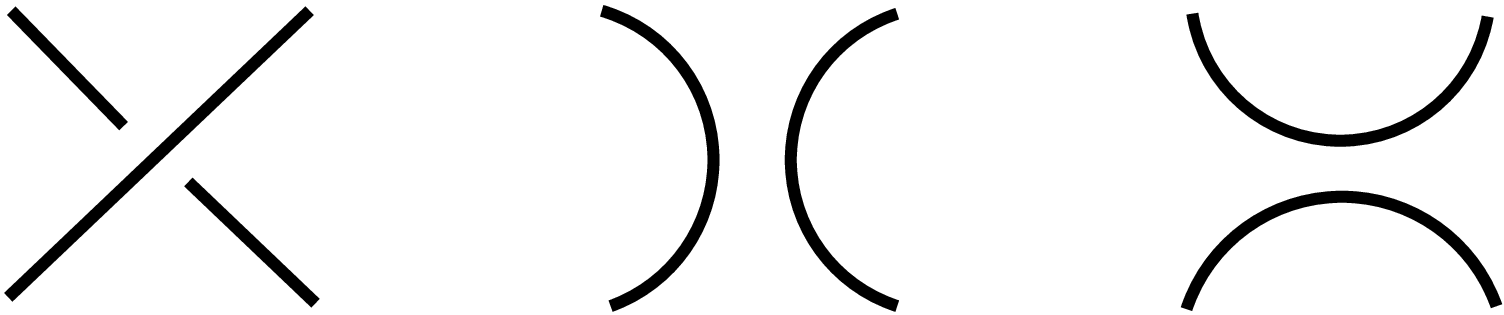} \\
{\hspace{0.2cm}$L$}\hspace{3cm}{$L_0$}\hspace{3cm}$L_{\infty}$
\end{center}
%\vspace{-3.5cm}
\caption{The diagram of the link $L$ with  crossing $c$ and its
smoothings $L_{0}$ and $L_{\infty}$ respectively.}\label{figure}
\end{figure}

Using this definition, one can easily see that any non-split alternating link is quasi-alternating at any crossing in any
reduced alternating diagram. The knot $8_{20}$ is the first example of quasi-alternating non-alternating knot in the knot table.
It  is worth mentioning here that it  is impossible to determine that a given link is not quasi-alternating using
the above recursive definition since one has to consider all crossings in all possible diagrams of this given link.

A different approach to address this problem
is to study the behavior of the invariants of quasi-alternating
links in order to find obstruction criteria for a link to be
quasi-alternating. Several such obstruction criteria have been introduced over the past fifteen years.
For instance, it was shown that these links are homologically thin  in  both Khovanov and link Floer homology \cite{MO,OS,ORS}.
In addition, the Heegaard Floer Homology of their  branched double  covers
depends only on the determinant of the link,
\cite{OS}.  On the other hand, a simple obstruction criterion
 has been introduced in terms of the degree of the $Q$-polynomial and the determinant of the link, \cite{QC}.
This obstruction has been sharpened by Teragaito in \cite{Te1}, then  extended by the same author
to the two-variable Kauffman polynomial \cite{Te2}.

A simple way to produce new examples of quasi-alternating links from old ones was introduced by Champanerkar and
Kofman \cite{CK}. Given a link $L$ with quasi-alternating diagram $D$ at a crossing $c$. Then any link diagram
obtained from $L$ by replacing the crossing $c$ by an alternating rational tangle of the same type is quasi-alternating
at any of the new crossings.

This construction was  generalized  to links obtained by replacing a crossing by a product of rational tangles  and
applied to study quasi-alternating Montesinos links \cite{QCQ}. In this paper, we investigate how does the Jones
polynomial interact with this twisting  property. More precisely, we prove that if the
Jones polynomial of a quasi-alternating link $L$ has no gaps, then so is the Jones polynomial of any  link $L^{*}$
obtained from it by replacing the quasi-alternating crossing by a product of rational tangles.
Based on this fact, in addition to  other computational  evidences, we conjecture that if $L$ is a prime quasi-alternating
link other than the $(2,n)$-torus link, then its Jones polynomial $V_L(t)$ has no gaps. It is well known that this condition
is satisfied by alternating links as it was proved in \cite{Th}.

This paper is organized as follows. In Section 2, the main theorem is proved and   a new obstruction for a link to be quasi-alternating
is conjectured.  This conjecture is proved for quasi-alternating Montesinos links in Section 3. In Section 4, we  prove that the Conjecture holds  for quasi-alternating links of braid index 3.

\section{Main Conjecture and Some Consequences}

The Jones polynomial $V_L(t)$ is an invariant of oriented links. It is a Laurent polynomial with integral coefficients that might be defined in several ways.
In this section, we shall briefly introduce this  polynomial and review some of its properties needed in the sequel. Let us first define the Kauffman bracket polynomial.

\begin{defn}
The Kauffman bracket polynomial is a function from unoriented link diagrams in the oriented plane to the ring of Laurent polynomials with integer coefficients in an indeterminate $A$. It maps a diagram $D$ to $\left\langle D\right\rangle\in \mathbb Z[A^{-1},A]$ and it is  defined by the following realtions:
\begin{enumerate}
\item $\left\langle \bigcirc \right\rangle=1$,
\item $\left\langle \bigcirc \cup D\right\rangle=(-A^{-2}-A^2)\left\langle D\right\rangle$,
\item $\left\langle L\right\rangle=A\left\langle L_0\right\rangle+A^{-1}\left\langle L_\infty\right\rangle$,
\end{enumerate}
where $\bigcirc$ denotes a trivial circle, and  $L,L_0, \text{and }L_\infty$ represent unoriented link diagrams that are identical except in a small region
 where they look as in Figure \ref{figure}.
\end{defn}
\begin{defn}
The Jones polynomial $V_L(t)$ of an oriented link $L$ is the Laurent polynomial in $t^{1/2}$ with integer coefficients defined by
\begin{equation*}
V_L(t)=((-A)^{-3w(D)}\left\langle D\right\rangle)_{t^{1/2}=A^{-2}}\in \mathbb Z[t^{-1/2},t^{1/2}],
\end{equation*}
where the writhe of the oriented diagram $D$, denoted by $w(D)$, is the number of crossings of the first type minus
the number of crossings of second type as pictured  in Figure \ref{Diagram1}.
\end{defn}

\begin{figure}[h]
	\centering
		\includegraphics[scale=0.12]{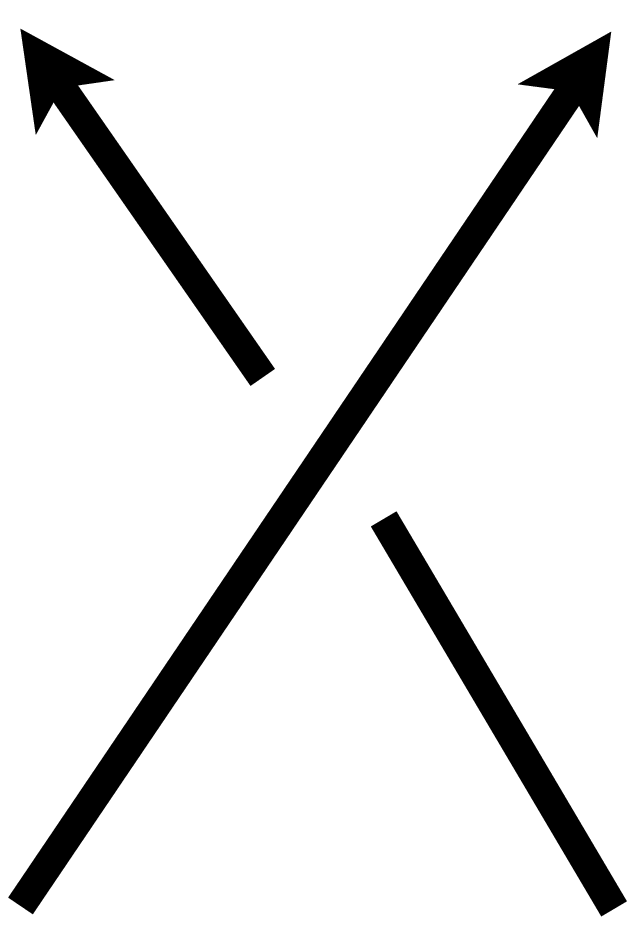}\hspace{1cm}\reflectbox{\includegraphics[scale=0.12]{Diagram1}}
	\caption{Crossings of the first and second type respectively}
	\label{Diagram1}
\end{figure}

We always can write the Jones polynomial of any link $L$ as follows:
\[
V_{L}(t)=t^{r}\displaystyle\sum_{i=0}^{m} a_it^i,
\text{where} \ m \geq 0, a_0 \neq 0  \ \text{and} \ a_m \neq 0.
\]
Therefore, the Kauffman bracket of the diagram $D$ of the link $L$ can be written as follows:
\[
\langle D \rangle(A)=A^{s}\displaystyle\sum_{i=0}^{m} a_iA^{4i},
\text{where} \ m \geq 0, a_0 \neq 0  \ \text{and} \ a_m \neq 0.
\]
We aim  to study  the Jones polynomial of quasi-alternating links. In  \cite{MO}, it was proved  that
the reduced ordinary Khovanov homology group of any quasi-alternating link is thin. Consequently, the Jones polynomial
of any quasi-alternating link is alternating. In other
words, its coefficients satisfy  $a_{i}a_{i+1} \leq 0$ for $i = 0,1,\ldots, m-1$. This inequality is strict  for prime
alternating links, other than $(2,n)$-torus links, as it was shown in \cite[Theorem\,1(iv)]{Th}. In this case,
we say that Jones polynomial has no gaps. It is easy to see that the Jones polynomial has no gaps if and only if
the Kauffman bracket satisfies $a_{i}a_{i+1} \neq 0$, where $i = 0,1,\ldots, m-1$.

It is known that quasi-alternating links  share the same homological properties with alternating links.
A natural question is to ask whether  we can  extend the above result about the Jones polynomial
to  the class of prime quasi-alternating links, other than $(2,n)$-torus links.
We conjecture the following.
\begin{conj}\label{main}
If $L$ is a prime quasi-alternating link, other than $(2,n)$-torus link, then the coefficients of the Jones polynomial
of $L$ satisfy  $a_{i}a_{i+1} < 0$ for all $0\leq i \leq m-1$.
\end{conj}

If this conjecture is true, then we will  get  a positive solution of  \cite[Conjecture\,3.8]{QC} that states that
$\det(L) \geq \br(L)$, where $\det(L)$ and $\br(L)$ denote the
determinant and the breadth of the Jones polynomial of the  quasi-alternating link $L$, respectively.
The last conjecture is a weaker version of \cite[Conjecture\,1.1]{QQJ} which  states that $\det(L) \geq c(L)$, where $c(L)$ denotes the
crossing number of the  quasi-alternating link $L$.

The following theorem makes use of the twisting construction introduced in \cite[Page.\,2452]{CK} which consists of
replacing a crossing by a rational tangle that extends it. This construction has been
generalized later to product of rational tangles in \cite[Def.\,2.5]{QCQ}.

\begin{thm}\label{new}
Let $L$ be  a quasi-alternating link at some crossing $c$ and let $L^{*}$ be the link  obtained from $L$ by replacing
the crossing $c$ by a product of rational tangles that extends $c$. If $V_L(t)$ has no gaps, then so is $V_{L^*}(t)$.
\end{thm}

Before proving the above theorem, we state the following lemma which is easy to prove.
\begin{lemm}\label{monomial}
Suppose $h$ is a product of two alternating polynomials then the monomial $x^{n}$ in $h$
has nonzero coefficient if at least one $a_{i}b_{n-i} \neq 0$ where $a_{i}$ is the coefficient of $x^{i}$ and $b_{n-i}$
is the coefficient of $x^{n-i}$ in the first and second polynomials respectively.
\end{lemm}
Now, we shall prove Theorem 2.4.
\begin{proof}
Let $L$ be a quasi-alternating link at the crossing $c$.
We may assume that $\sign(c) > 0$ as shown in Figure \ref{figure} by taking
the mirror image if it is needed. For a positive integer $n$,
we let $L^{n}$ denote the link diagram with $n$  vertical or horizontal positive half-twists
crossings replacing the crossing $c$.

To prove our result, we first show that $\langle L^{n} \rangle$ satisfies
$a_{i}a_{i+1} \neq 0$ if $V_{L}(t)$ satisfies $a_{i}a_{i+1} \neq 0$ for $i=0,1,\ldots, m-1$.
It is a simple exercise to prove that:
\begin{align*}
\langle L^{n} \rangle & = A^{n} \langle L_{0} \rangle + \left(\sum_{i=0}^{n-1}(-1)^{i}
A^{n-4i-2}\right)\langle L_{\infty} \rangle ,\\
\langle L^{n} \rangle & = \left(\sum_{i=0}^{n-1}(-1)^{i}A^{n-4i-2}\right)\langle L_{0}
\rangle (A) + A^{n} \langle L_{\infty} \rangle,
\end{align*}
for the case of $n$ vertical and $n$ horizontal crossings respectively. Now, we shall  discuss
the first case. The second case can
be treated in a similar manner.
\begin{align*}
\langle L^{n} \rangle & = A^{n} \langle L_{0} \rangle + \left(\sum_{i=0}^{n-1}
(-1)^{i}A^{n-4i-2}\right)\langle L_{\infty} \rangle ,\\
& = A^{n-1}\left(A \langle L_{0} \rangle + A^{-1}\langle L_{\infty} \rangle \right)
+ \left(\sum_{i=1}^{n-1}(-1)^{i}A^{n-4i-2}\right)\langle L_{\infty} \rangle\\
& = A^{n-1} \langle L \rangle + \left(\sum_{i=1}^{n-1}(-1)^{i}A^{n-4i-2}\right)
\langle L_{\infty} \rangle
\end{align*}

Note that no monomial of the first term will cancel out with a monomial of the second term
%$\left(\sum_{i=1}^{n-1}(-1)^{i}A^{n-4i-2}\right)\langle L_{\infty} \rangle$
because $\det(L^{n}) = \det(L) + (n-1) \det(L_{\infty})$ which can be proved
since $L^{n}$ is quasi-alternating as a result of \cite[Theorem\,2.1]{CK}.

Moreover, if $A^{t}$ is a monomial of nonzero coefficient in $\langle L_{\infty} \rangle$, then $A^{n+t-6}$
is a monomial with nonzero coefficient in the second term and at the same time
the monomial $A^{n+t-2}$ has nonzero coefficient in the first term. Note that these two monomials
will not cancel out in $\langle L^{n} \rangle$ because of the above note. This proves
that if $a_{i}a_{i+1} = 0$ then either this happens in first term or in the second term.

The first case is impossible because of the assumption on the link $L$.
For the second one, suppose the monomial of nonzero coefficient with the
lowest degree in the second term such that
$a_{i}a_{i+1} = 0$ is $A^{4t+s}$ for some $s$ and $t$. In other words, the monomial $A^{4t+s}$ has nonzero
coefficient and $A^{4t+s+4}$ has zero coefficient in the second term and $t$ is the smallest such
number. We show that the coefficient of the monomial $A^{4t+s+4}$ has nonzero coefficient in the first
term. As a result of Lemma \ref{monomial}, having a nonzero coefficient of the monomial $A^{4t+s}$ implies that at least one
of the following monomials $A^{4t+s-n+6}, A^{4t+s-n+10}, \ldots, A^{4t+s+3n-2}$ has to have nonzero coefficient
in $\langle L_{\infty} \rangle$ and at the same time having a zero coefficient of the monomial $A^{4t+s+4}$
implies that all of the following monomials have to have zero coefficient
$A^{4t+s-n+10}, A^{4t+s-n+14}, \ldots, A^{4t+s+3n+2}$
in $\langle L_{\infty} \rangle$. It is clear in this case that the monomial $A^{4t+s-n+6}$ has nonzero coefficient
and all of the other monomials have to have zero coefficient in $\langle L_{\infty} \rangle$. This will imply
that the monomial $A^{4t+s+4}$ has nonzero coefficient in the first term.

Now this proves that $\langle L^{n} \rangle$ satisfies $a_{i}a_{i+1}\neq 0$ for $i= 0, 1,\ldots, m-1$.
The fact that the link $L^{n}$ is quasi-alternating implies that $\langle L^{n} \rangle$ is alternating.
Therefore, $\langle L^{n} \rangle$ satisfies $a_{i}a_{i+1}< 0$ for $i= 0, 1,\ldots, m-1$. Finally the result
follows since any product of rational tangles can be obtained by a sequence of integer tangles.

\end{proof}

\begin{coro}
If a link is  quasi-alternating  with a gap in its  Jones polynomial, then it cannot be obtained by twisting a quasi-alternating link with no gaps in its Jones polynomial.
\end{coro}

\begin{rem}
With the aid of \cite{CL} and the tables of quasi-alternating knots of at most 12 crossings in \cite{J},
we checked by hand that all knots with   12 crossings or less with gap
in the Jones polynomial  are either  $(2,n)$-torus knots or not quasi-alternating. This confirms the above
conjecture for all knots of at most 12 crossings. As a consequence of the above theorem, we conclude that a counter example
to the above conjecture has to be a knot of at least 13 crossings that can not be obtained by twisting a quasi-alternating
knot of at most 12 crossings.
\end{rem}
%%%%%%%%%%%%%%%%%%%%%%%%%%%%%%%%%%%%%%%%%%%%%%%%%%%%%%%%%%%%%%%%%%%%%%%%%%%%%%%%%%%%%%%%%%%%%%%%%%%%%%%%
\section{The Jones polynomial  of Quasi-alternating Montesinos Links}
%%%%%%%%%%%%%%%%%%%%%%%%%%%%%%%%%%%%%%%%%%%%%%%%%%%%%%%%%%%%%%%%%%%%%%%%%%%%%%%%%%%%%%%%%%%%%%%%%%%%%%%%

In this section, we prove Conjecture \ref{main} for all quasi-alternating Montesinos links.
Let $\alpha$ and $\beta$ be coprime integers with $\alpha > \beta \geq 1$ and $[a_{1},a_{2},\ldots,a_{n}]$ be the continued
fraction of the rational number $\frac{\beta}{\alpha}$ of positive integers. In \cite{KaLo}, Kauffman and Lopez
defined a sequence of positive integers recursively by
\[
T(0) = 1, T(1) = a_{1}, T(m) = a_{m}T(m-1) + T(m-2),
\]
and showed that $T(n)$ is equal to the determinant of the rational
link $L_{\alpha,\beta}$ of slope $\frac{\beta}{\alpha}$.

\begin{lemm}
For any positive rational number $\frac{\beta}{\alpha}$, there is a
continued fraction $[a_{1},a_{2},\ldots,a_{n}]$ of positive
integers.
\end{lemm}
\begin{lemm}
Let $[a_{1},a_{2},\ldots,a_{n}]$ be a continued fraction of
$\frac{\beta}{\alpha}$ of positive integers, then $T(n - 1) =
\beta$ and $T(n) = \alpha$.
\end{lemm}
\begin{proof}
Using the recursive formula for $T(n)$, we can see that if  $d|T(n)$ and
$d|T(n-1)$, then $d|T(n-2)$. If this process is repeated, we obtain
that $d|T(0)=1$. This shows that $T(n)$ and $T(n-1)$ are relatively prime.

Now, we have

\begin{align*}
\frac{T(n)}{T(n-1)} & = \frac{a_{n}T(n-1) + T(n-2)}{T(n-1)} \\
& = a_{n} + \frac{T(n-2)}{T(n-1)} \\
& = a_{n} + \frac{1}{\frac{T(n-1)}{T(n-2)}}\\
& = a_{n} + \frac{1}{a_{n-1} + {\frac{T(n-3)}{T(n-2)}}}\\
&  \ \ \vdots  \\
& = a_{n} + \cfrac{1}{a_{n-1}+\cfrac{1}{\ddots + \cfrac{1}{a_{1} }}}
= \frac{\alpha}{\beta}.
\end{align*}

\end{proof}

\begin{lemm}
Let $[a_{1},a_{2},\ldots,a_{n}]$ be a continued fraction of
$\frac{\beta}{\alpha}$ of positive integers, then $\alpha \geq a_{1}+a_{2}+\ldots+a_{n}$.
\end{lemm}
\begin{proof}
We use induction on the length of the continued fraction. It is clear that this holds for $n=1$. Now
suppose this holds for any continued fraction of length $n-1$. Therefore, we have $T(n-1) \geq
a_{1}+a_{2}+\ldots +a_{n-1}$. Now $\alpha = a_{n}T(n-1)+T(n-2) \geq a_{n}T(n-1) + 1 \geq a_{n}+T(n-1) =
a_{n}+a_{n-1}+\ldots+a_{1}$. The last inequality follows if we apply the result for $n=2$.
\end{proof}

\begin{defn}
A Montesinos link is a link that has a diagram as shown in Figure
\ref{figure1}(a). In this diagram, $e$ is an integer that represents
the number of half twists and the box \fbox{$\alpha_{i}, \beta_{i}$}
stands for a rational tangle of slope
$\frac{\beta_{i}}{\alpha_{i}}$, where  $\alpha_{i} > 1$ and
$\beta_{i}$ are coprime integers.
\end{defn}

\begin{figure}[htbp]
    \begin{center}
\includegraphics[width=6cm, height=5cm]{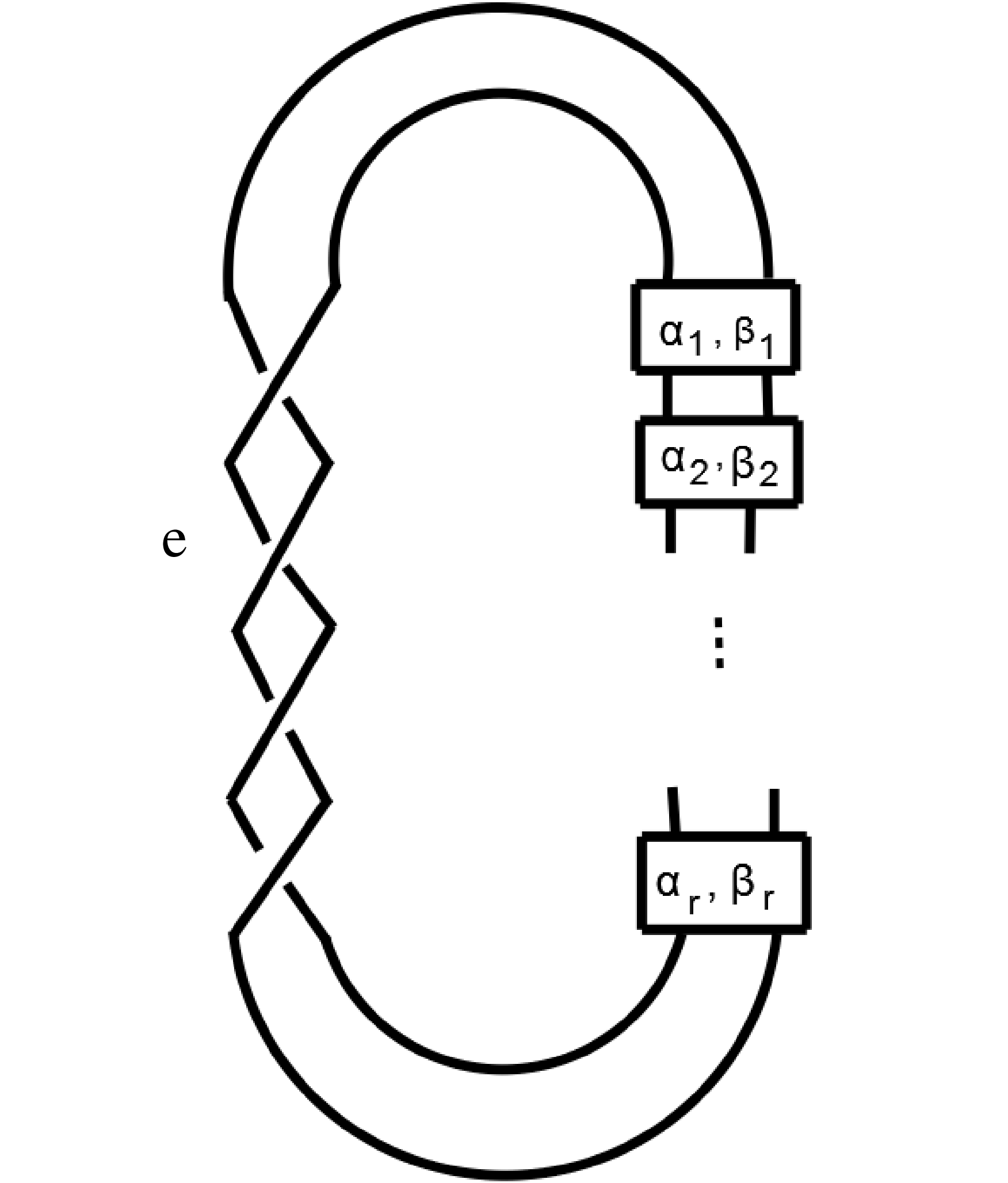} \hspace{2cm}
  \includegraphics[width=4cm, height =3cm]{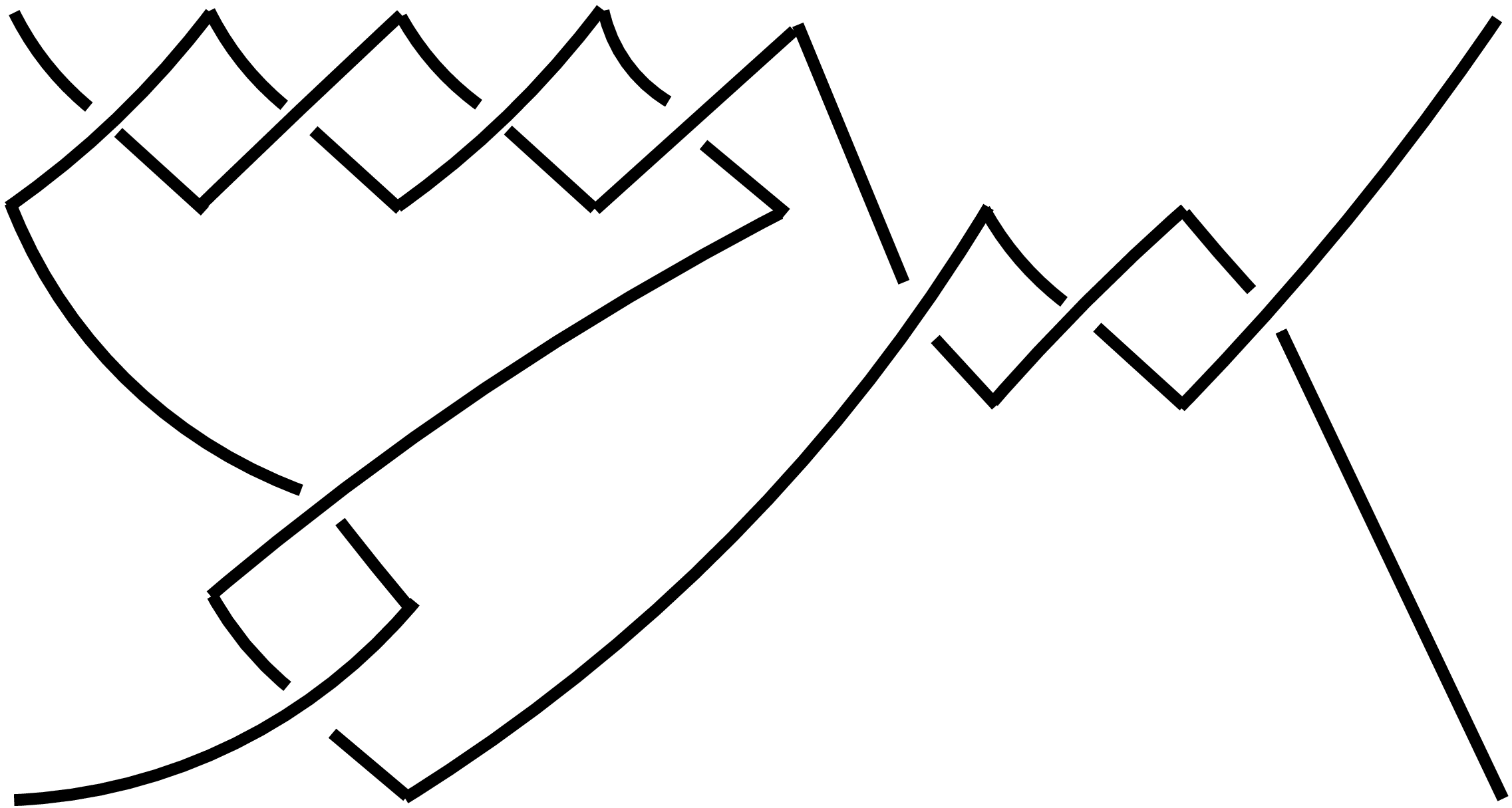}\\
(a)\hspace{6cm} (b)
\end{center}
\caption{Montesinos link diagram in (a) with $e = 4$. (b) represents
the rational tangle of slope $\frac{-7}{31}$ with associated
continued fraction $[-3,-2,-4]$.}\label{figure1}
\end{figure}

It is worth mentioning here that our conventions for the signs in the picture above are the same as in \cite{I} and \cite{QCQ}. Recall that    Montesinos links are classified by the following theorem.
\begin{thm}\cite{Bo,BoSi}
The Montesinos link $M(e;(\alpha_{1},\beta_{1}),
(\alpha_{2},\beta_{2}), \dots, (\alpha_{r},\beta_{r}))$ with $r \geq
3$ and $\sum_{j = 1}^{r} \frac{1}{\alpha_{j}} \leq r -2$, is
classified by the ordered set of fractions
$(\frac{\beta_{1}}{\alpha_{1}}\pmod
1,\frac{\beta_{2}}{\alpha_{2}}\pmod 1, \ldots,
\frac{\beta_{r}}{\alpha_{r}}\pmod 1)$ up to cyclic permutations and
reversal of order, together with the rational number $e_{0} = e -
\sum_{j = 1}^{r} \frac{\beta_{j}}{\alpha_{j}}$.
\end{thm}

\begin{defn}
A Montesinos link $M(e;(\alpha_{1},\beta_{1}),
(\alpha_{2},\beta_{2}), \dots, (\alpha_{r},\beta_{r}))$ is in
standard form if $0 < \frac{\beta_{i}}{\alpha_{i}} < 1$ for $ i =
1,2, \ldots, r$.
\end{defn}
\begin{rem}
According to the above theorem, every Montesinos link has a unique
standard form.
\end{rem}

\begin{lemm}\label{simple1}
The Montesinos link $L=M(e;(\alpha_{1},\beta_{1}),(\alpha_{2},\beta_{2}),\ldots,(\alpha_{r},\beta_{r}))$ in standard form
with $e < 0$ is not a $(2,n)$-torus link for any integer $n$.
\end{lemm}
\begin{proof}
Suppose that the link $L$ is a $(2,n)$-torus link for some integer $n$. It is clear that the corresponding Montesinos diagram of $L$
is reduced alternating diagram, so this diagram has minimal number of crossings among all other diagrams which has to be
equal to $n$ from the assumption. Therefore, we have $n = \sum_{j=1}^{r}\sum_{i=1}^{s_{j}} a_{i}^{j} $ with
$[a_{1}^{j},a_{2}^{j},\ldots,a_{s_{j}}^{j}]$ being the continued fraction of $\frac{\alpha_{j}}{\beta_{j}}$. Using the fact that $\det(L) = \left(\prod_{j=1}^{r} \alpha_{j}\right) \left(-e +
\sum_{j=1}^{r} \frac{\beta_{j}}{\alpha_{j}}\right)$, we get:
\[
n = \det(L) = \left(\prod_{j=1}^{r} \alpha_{j}\right) \left(-e +
\sum_{j=1}^{r} \frac{\beta_{j}}{\alpha_{j}}\right) > \prod_{j=1}^{r} \alpha_{j} > \sum_{j=1}^{r} \alpha_{j} > \sum_{j=1}^{r}
\sum_{i=1}^{s_{j}} a_{i}^{j}  = n,
\]
which is impossible. Thus, $L$ cannot be a  $(2,n)$-torus link.
\end{proof}

\begin{lemm}\label{simple2}
The Montesinos link $L=M(0;(\alpha_{1},\beta_{1}),(\alpha_{2},\beta_{2}))$ with $\frac{\alpha_{2}}{\alpha_{2} - \beta_{2}} >
\frac{\alpha_{1}}{\beta_{1}}$ is not a $(2,n)$-torus link for any integer $n$.
\end{lemm}

\begin{proof}
It is easy to see that $\frac{\alpha_{2}}{\alpha_{2} - \beta_{2}} >
\frac{\alpha_{1}}{\beta_{1}}$ is equivalent to say $\frac{\beta_{1}}{\alpha_{1}} + \frac{\beta_{2}}{\alpha_{2}} > 1$.
We know that $\alpha_{1},\alpha_{2} \geq 2$.
Hence, we conclude that either $\beta_{1}$ is not equal to 1 or $\beta_{2}$ is not be equal to 1.
Now suppose that the link $L$ is a $(2,n)$-torus link for some integer $n$. It is clear that the corresponding Montesinos diagram of $L$
is reduced alternating diagram, so this diagram has minimal number of crossings among all other diagrams which has to be
equal to $n$ from the assumption. Therefore, we have $n = a_{1}+a_{2}+\ldots+a_{s}+b_{1}+b_{2}+\ldots+b_{t}$ with
$[a_{1},a_{2},\ldots,a_{s}], [b_{1},b_{2},\ldots,b_{t}]$ being the continued fractions of $\frac{\alpha_{1}}{\beta_{1}}$ and
$\frac{\alpha_{2}}{\beta_{2}}$ of positive integers, respectively. Also we have
\begin{align*}
n = \det(L) =
\alpha_{1}\beta_{2}+\alpha_{2}\beta_{1} & \geq \beta_{2}(a_{1}+a_{2}+\ldots+a_{s}) + \beta_{1}(b_{1}+b_{2}+\ldots+b_{t})\\
& > (a_{1}+a_{2}+\ldots+a_{s}) + (b_{1}+b_{2}+\ldots+b_{t}) =n,
\end{align*}
which is impossible. In conclusion, $L$ cannot be a  $(2,n)$-torus link.

\end{proof}

For the next result, we need the following theorem that is obtained by combining \cite[Theorem\,3.5]{QCQ}
and \cite[Theorem\,1]{I}.
\begin{thm}
The Montesinos link $ L =
M(e;(\alpha_{1},\beta_{1}),(\alpha_{2},\beta_{2}),
\ldots,(\alpha_{r},\beta_{r}) )$ in standard form is
quasi-alternating iff one of the following conditions is satisfied
\begin{enumerate}
\item $ e \leq 0$;
\item $e \geq r$;
\item $e = 1$  with $\frac{\alpha_{i}}{\alpha_{i} - \beta_{i}} > \min \{ \frac{\alpha_{j}}{\beta_{j}}\ | \  j \neq i \}$
for some $1 \leq i \leq r$;
\item $e = r - 1$  with $\frac{\alpha_{i}}{\beta_{i}} > \min \{ \frac{\alpha_{j}}{\alpha_{j} - \beta_{j}}\ | \  j \neq i \}$
for some $1 \leq i \leq r$.
\end{enumerate}
\end{thm}

\begin{thm}
The coefficients of the Jones polynomial of any quasi-alternating Montesinos link satisfy  $a_{i}a_{i+1} < 0$ for all $0\leq i \leq m-1$.
\end{thm}

\begin{proof}
The result holds for the first two cases since the Montesinos link in standard form in these two cases
is an alternating link that is not a $(2,n)$-torus link
for any integer $n$ according to Lemma \ref{simple1}. For the third case, the Montesinos link
can be obtained by replacing the crossing in the middle tangle in
$M(0;(\alpha_{1},\beta_{1}),(1,1),(\alpha_{i},\beta_{i}-\alpha_{i}))$ with $\frac{\alpha_{i}}{\alpha_{i} -\beta_{i}}
> \frac{\alpha_{1}}{\beta_{1}}$ by a product of rational tangles that extends it according to \cite[Theorem\,3.5(3)]{QCQ}.
Now we check that $L=M(0;(\alpha_{1},\beta_{1}),(1,1),(\alpha_{i},\beta_{i}-\alpha_{i}))$ has no gap in the Jones polynomial.
It is enough to show that it is alternating. We know that $L$ is equivalent to
$M(0;(1,1),(\alpha_{1},\beta_{1}),(\alpha_{i},\beta_{i}-\alpha_{i}))$ and this one is equivalent to
$M(-1;(\alpha_{1},\beta_{1}),(\alpha_{i},\beta_{i}-\alpha_{i}))$. Now the last link is equivalent
to the link $M(0;(\alpha_{1},\beta_{1}),(\alpha_{i},\beta_{i}))$ which is alternating since the corresponding
diagram is alternating. The last link is not a $(2,n)$-torus link according to Lemma \ref{simple2}, so its Jones polynomial has no gap.
The last case follows using the same argument after we put the link in standard form. Finally,
the result follows as a direct consequence of Theorem \ref{new}.
\end{proof}

\begin{rem}
The above theorem can be specialized to the case of pretzel links. But rather than using \cite[Theorem\,1]{I}
and \cite[Theorem\,3.5(3)]{QCQ}, we use \cite[Theorem\,1]{G} and \cite[Corollary\,3.6(4)]{QCQ}.
\end{rem}

%%%%%%%%%%%%%%%%%%%%%%%%%%%%%%%%%%%%%%%%%%%%%%%%%%%%%%%%%
\section{The Jones polynomial  of Quasi-alternating closed 3-braids}
%%%%%%%%%%%%%%%%%%%%%%%%%%%%%%%%%%%%%%%%%%%%%%%%%%%%%%%%%
In this section, we prove that Conjecture \ref{main} holds for prime quasi-alternating links with braid index equal to 3.
Quasi-alternating links of braid index 3 have been classified by Baldwin \cite{B} based on Murasugi's  classification of 3-braids \cite{Mu}.
For $n\geq 2$, let $B_n$ be the  braid group on $n$ strings.  This  group is  generated by the elementary braids
$\sigma_{1}, \sigma_{2}, \dots ,\sigma_{n-1}$ subject to the
following relations:
\begin{align*}
\sigma_{i} \sigma_{j} & =\sigma_{j} \sigma_{i} \mbox{ if } |i-j| \geq 2\\
\sigma_{i} \sigma_{i+1} \sigma_{i} & = \sigma_{i+1}
\sigma_{i}\sigma_{i+1}, \ \forall \ 1 \leq i \leq n-2.
\end{align*}

The two generators $\sigma_1$  and $\sigma_2$ of the braid group
$B_3$ are pictured in Figure \ref{braid}. Recall that every link $L$
in $S^3$ can be obtained as the closure of a certain braid $b$, denoted by $L=\hat b$.
It is well known that  3-braids have been classified, up to
conjugation,  by Murasugi \cite{Mu}.
\begin{figure}
\begin{center}
\includegraphics[width=5cm,height=1.5cm]{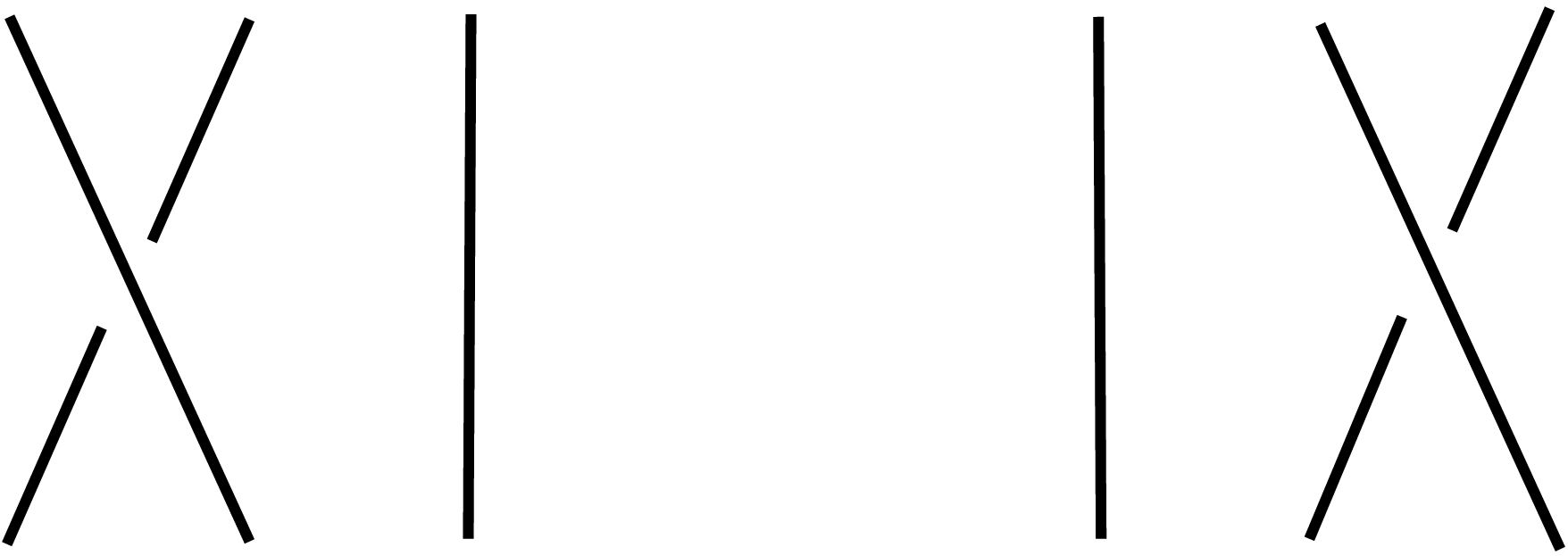}
\caption{The generators $\sigma_1$  and $\sigma_2$ of $B_3$
respectively} \label{braid}
\end{center}
\end{figure}

\begin{thm}\label{thm}
Let $b$ be a 3-braid and let $h=(\sigma_1\sigma_2)^3$ be a full
positive twist. Then $b$ is conjugate to exactly one of the
following:
\begin{enumerate}
\item $h^n\sigma_1^{p_1}\sigma_2^{-q_1}\dots \sigma_1^{p_s}\sigma_2^{-q_s}$, where $s,p_i$ and $q_i$ are positive integers.
\item $h^n\sigma_2^{m}$ where $m \in \Bbb{Z}$.
\item $h^n\sigma_1^{m}\sigma_2^{-1}$, where $m \in \{-1,-2,-3\}$.
\end{enumerate}
\end{thm}
Baldwin classified quasi-alternating closed 3-braids as  in the
following theorem \cite{B}:

\begin{thm}\label{thm}
Let $L$ be a closed 3-braid, then
\begin{enumerate}
\item If $L$ is the closure of $h^n\sigma_1^{p_1}\sigma_2^{-q_1}\dots \sigma_1^{p_s}\sigma_2^{-q_s}$,
where $s, p_i$ and $q_i$ are positive integers, then $L$ is
quasi-alternating if and only if $n \in \{-1,0,1\}$.
\item If $L$ is the closure of $h^n\sigma_2^{m}$, then $L$ is quasi-alternating if and only if
either $n=1$ and $m \in \{-1,-2,-3\}$ or $n=-1$ and $m \in
\{1,2,3\}$.
\item If $L$ is the closure of $h^n\sigma_1^{m}\sigma_2^{-1}$ where $m \in \{-1,-2,-3\}$.
Then $L$ is quasi-alternating  if and only if $n \in \{0,1\}$.
\end{enumerate}
\end{thm}

Now we state the main theorem in this section.
\begin{thm}
The coefficients of the Jones polynomial of any prime quasi-alternating link of braid index 3 satisfy
$a_{i}a_{i+1} < 0$ for all $0\leq i \leq m-1$.
\end{thm}
\begin{proof}
The Jones polynomial of a closed 3-braid can be computed using a relatively simple formula introduced by
Birman \cite{Bi}. Given a 3-braid  $\alpha$,  let $e_\alpha$ be the exponent sum of $\alpha$ as a word in the
elementary braids $\sigma_1$ and $\sigma_2$.   Let $\psi_t: B_{3}
\longrightarrow GL(2, \mathbb{Z}[t,t^{-1}])$ be the Burau
representation defined on the generators of $B_3$ by $\psi_t(\sigma_1)=\left [\begin{array}{cc}
-t&1\\
0&1
\end{array} \right ] $ and $\psi_t(\sigma_2)=\left [\begin{array}{cc}
1&0\\
t&-t
\end{array} \right ] $. 

Then, the Jones polynomial of the link $\hat\alpha$ is given by the following formula
 $$
 V_{\hat\alpha}(t)=(-\sqrt{t})^{e_\alpha}(t+t^{-1}+tr(\psi_t(\alpha))),$$

where   $tr$ denotes the usual matrix-trace function.\\
Let us consider  the Jones polynomial of  a braid of type 1 in Baldwin's Theorem. If $n=0$,
then the link is alternating and the Conjecture holds. We will now consider the case $n=1$. The case
$n=-1$ is treated in a similar way.  Assume that $\alpha=h\beta$, where
$\beta=\sigma_1^{p_1}\sigma_2^{-q_1}\dots \sigma_1^{p_s}\sigma_2^{-q_s}$.  Since $\psi_t(h)=t^3I_2$,
then  $tr(\psi_t(\alpha))=t^3tr(\psi_t(\beta))$. For any positive integers $p_i$ and $q_i$, elementary computations show that:

$\psi_t(\sigma_1^{p_i})=\left [\begin{array}{cc}
(-t)^{p_i}&\displaystyle\frac{1-(-t)^{p_i}}{1+t}\\
&\\
0&1
\end{array} \right ] $, and
$\psi_t(\sigma_2^{-q_i})=\left [\begin{array}{cc}
1&0\\
&\\
\displaystyle\frac{1-(-t)^{-q_i}}{1+\displaystyle\frac{1}{t}}&(-t)^{-q_i}
\end{array} \right ] $.

Consequently

 $$\psi_t(\sigma_1^{p_i}\sigma_2^{-q_i})=
\left [\begin{array}{cc}
(-t)^{p_i}+\displaystyle\frac{1-(-t)^{p_i}}{1+t}\frac{1-(-t)^{-q_i}}{1+\displaystyle\frac{1}{t}}& (-t)^
   {-q_i}\displaystyle\frac{1-(-t)^{p_i}}{1+t}\\
   &\\
\displaystyle\frac{1-(-t)^{-q_i}}{1+\displaystyle\frac{1}{t}}& (-t)^{-q_i}
\end{array} \right ] $$

Note that the matrix  $\psi_t(\sigma_1^{p_i}\sigma_2^{-q_i})$ above   is of the form $\left [\begin{array}{cc}
A_{p_i,q_i}(t)&B_{p_i,q_i}(t)\\
C_{p_i,q_i}(t)&D_{p_i,q_i}(t)
\end{array} \right ]$, where the four entries are Laurent polynomials, each of which  is  alternating and has
no gaps. Moreover, in each of these polynomials the coefficient of the monomial $t^i$ has the same sign as $(-1)^{i}$.
Since $\beta$ is a product of braids of the form $\sigma_1^{p_i}\sigma_2^{-q_i}$, then each  entry of the matrix
$\psi_t(\beta)$ is indeed a sum of  products of elements of type $A_{p_i,q_i}(t), B_{p_i,q_i}(t),
C_{p_i,q_i}(t)$ and $D_{p_i,q_i}(t)$. Consequently, each of the entries of  $\psi_t(\beta)$ is alternating
with no gaps and for each $i$ the sign of the coefficient of the  monomial  $t^i$ is the same as $(-1)^i$.
Moreover, one can see that  $tr(\psi_t(\beta))$ is a Laurent polynomial which is alternating. When considering the sum
of diagonal entries of $\psi_t(\beta)$, no cancellation will happen  because  of the observation above about the
sign of each monomial.

Consequently,  the polynomial $tr(\psi_t(\beta))$  will be alternating with no gaps. In addition, from the computation above,
we can  see that the maximum power of $t$ that appears  in $tr(\psi_t(\beta))$  is $p=p_1+p_2+\dots+p_s$, while
the minimum power is $-(q_1+q_2\dots+q_s)=-q$. If we consider the expression $t+t^{-1}+t^3\psi_t(\beta)$ then
it is  clear that no gaps will appear in this expression if $q>2$ because the sign of the monomials $t$ and
$t^{-1}$ is positive in $t^3\psi_t(\beta)$. Moreover, the minimal power in $t^3\psi_t(\beta)$ is $-q+3 \leq 0$.

It remains now to check  the cases $q=1$ and $q=2$. In these cases, the Jones polynomials will have gaps. Indeed,
the constant term in $t^{-1}+t+t^3\psi_t(\beta)$ will be zero. We  shall  prove that in both  cases
the link $\hat\alpha$ will be either a $(2,n)$-torus link or a connected sum  of such links. Two  cases are to be considered.\\
If $q=1$, then
$$\begin{array}{rl}
(\sigma_1\sigma_2)^3\sigma_1^{p_1}\sigma_2^{-1}&\equiv
\sigma_1\sigma_2 \sigma_1\sigma_2\sigma_1\sigma_2 \sigma_1^{p_1}
\sigma_2^{-1}\\&\equiv \sigma_2\sigma_1
\sigma_2\sigma_1\sigma_2\sigma_1 \sigma_1^{p_1} \sigma_2^{-1}\\
& \equiv \sigma_1\sigma_2\sigma_1\sigma_2\sigma_1 \sigma_1^{p_1} \\
& \equiv \sigma_1\sigma_1\sigma_2\sigma_1\sigma_1^{p_1+1}\\
& \equiv \sigma_1^{p_1+4}.\\
\end{array}$$
where the symbol$\equiv$ is used to indicate that the closure of the braids are the same.  Thus,
$\hat\alpha$ is the torus link of type $(2,p_1+4)$.
If $q=2$, then   there are two subcases to be considered. If
$\alpha=(\sigma_1\sigma_2)^3\sigma_1^{p_1}\sigma_2^{-2}$, we have:
$$\begin{array}{rl}
(\sigma_1\sigma_2)^3\sigma_1^{p_1}\sigma_2^{-2}&\equiv
\sigma_1\sigma_2 \sigma_1\sigma_2\sigma_1\sigma_2 \sigma_1^{p_1}
\sigma_2^{-2}\\
&\equiv \sigma_2\sigma_1
\sigma_2\sigma_1\sigma_2\sigma_1 \sigma_1^{p_1} \sigma_2^{-2}\\
& \equiv \sigma_1\sigma_2\sigma_1\sigma_2\sigma_1 \sigma_1^{p_1}\sigma_2^{-1} \\
& \equiv \sigma_2\sigma_1\sigma_2\sigma_2\sigma_1 \sigma_1^{p_1}\sigma_2^{-1} \\
& \equiv \sigma_1\sigma_2\sigma_2\sigma_1 \sigma_1^{p_1} \\
& \equiv \sigma_2^2\sigma_1^{p_1+2}\\
\end{array}$$

Hence, $\hat\alpha$ is the connected sum of the torus link of type $(2,p_1+2)$ and the Hopf link.
The second subcase concerns the closure of
$\alpha=(\sigma_1\sigma_2)^3\sigma_1^{p_1}\sigma_2^{-1}\sigma_1^{p_2}\sigma_2^{-1}$.
Similar elementary calculations  in the braid group show that:

$$\begin{array}{rl}
(\sigma_1\sigma_2)^3\sigma_1^{p_1}\sigma_2^{-1}\sigma_1^{p_2}\sigma_2^{-1}&\equiv
\sigma_1\sigma_2 \sigma_1\sigma_2\sigma_1\sigma_2 \sigma_1^{p_1}
\sigma_2^{-1}\sigma_1^{p_2}\sigma_2^{-1}\\
&\equiv \sigma_2\sigma_1
\sigma_2\sigma_1\sigma_2\sigma_1 \sigma_1^{p_1}\sigma_2^{-1}\sigma_1^{p_2}\sigma_2^{-1}\\
& \equiv \sigma_2\sigma_1\sigma_2 \sigma_1^{p_1+1}\sigma_2^{-1}\sigma_1^{p_2+1} \\
& \equiv \sigma_1\sigma_2\sigma_1 \sigma_1^{p_1+1}\sigma_2^{-1}\sigma_1^{p_2+1} \\
& \equiv \sigma_2 \sigma_1^{p_1+2}\sigma_2^{-1}\sigma_1^{p_2+2} \\
\end{array}$$

Finally, it is not difficult to see that the closure of $\sigma_2 \sigma_1^{p_1+2}\sigma_2^{-1}\sigma_1^{p_2+2}$
is the connected sum of the torus links of type $(2,p_1+2)$ and $(2,p_2+2)$, see Figure \ref{connectedsumtorus}.

 \begin{figure}
\begin{center}
\includegraphics[width=3cm,height=4.5cm]{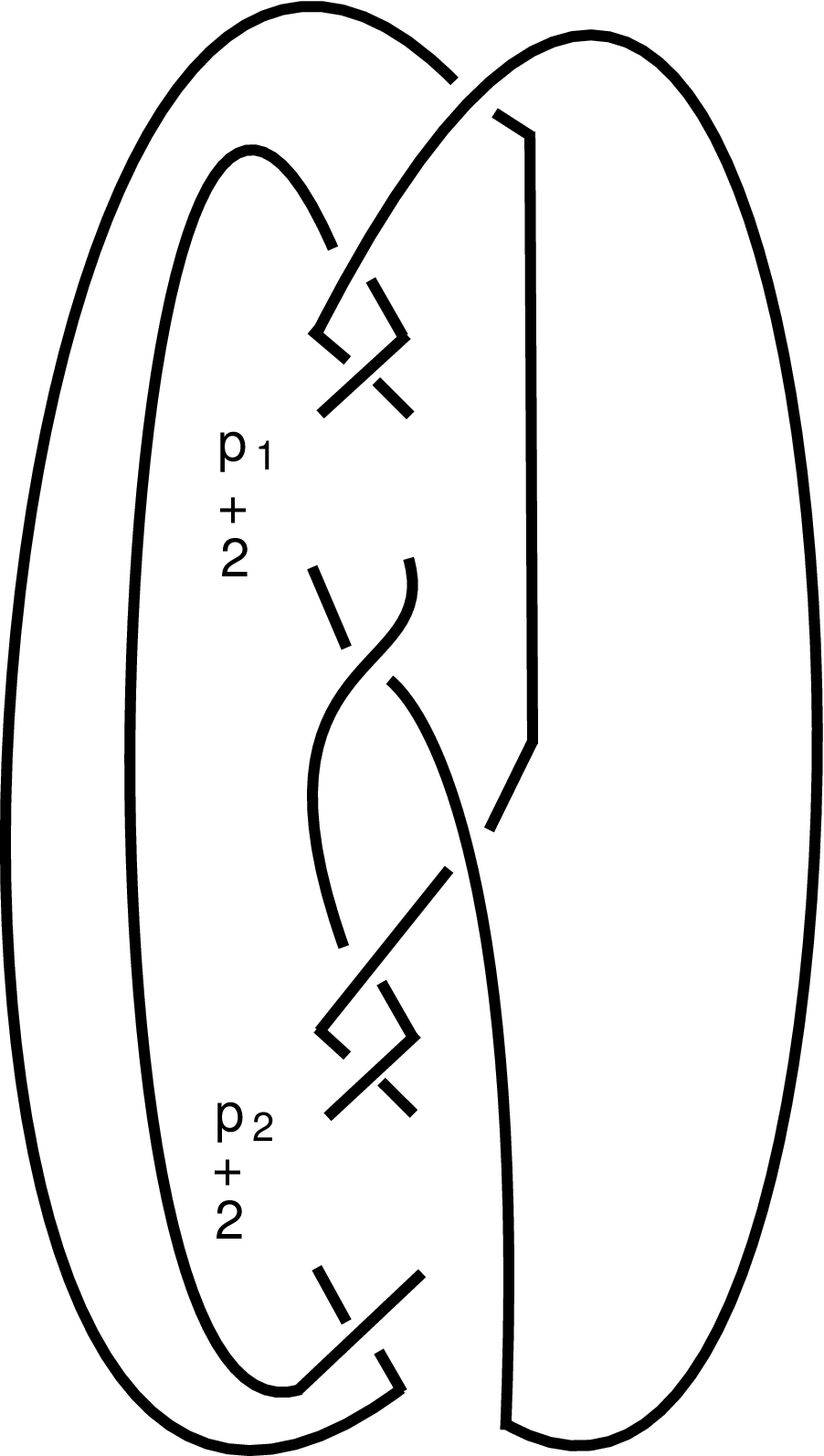}
\caption{The closure of braid $\sigma_2 \sigma_1^{p_1+2}\sigma_2^{-1}\sigma_1^{p_2+2}$} \label{connectedsumtorus}
\end{center}
\end{figure}
The conjecture then holds for closed 3-braids as in  the first case of Baldwin's Theorem, with $n=1$. For the case $n=-1$,
the proof can be done in a similar way. It is easy to check that cases 2 and 3 in Baldwin's Theorem yield either a torus
link of type $(2,n)$ for some $n$ or a connected sum of two links of this type.
\end{proof}

\end{document}